\documentclass[12pt,reqno]{amsart}

\usepackage{times}
\usepackage[T1]{fontenc}
\usepackage{mathrsfs}
\usepackage{latexsym}
\usepackage[dvips]{graphics}
\usepackage{epsfig}
\usepackage{amsmath,amsfonts,amsthm,amssymb,amscd}
\input amssym.def
\input amssym.tex
\usepackage{color}

\newcommand\be{\begin{equation}}
\newcommand\ee{\end{equation}}
\newcommand\bea{\begin{eqnarray}}
\newcommand\eea{\end{eqnarray}}
\newcommand\bi{\begin{itemize}}
\newcommand\ei{\end{itemize}}
\newcommand\ben{\begin{enumerate}}
\newcommand\een{\end{enumerate}}

\newtheorem{thm}{Theorem}[section]

\newtheorem{cor}[thm]{Corollary}
\newtheorem{lem}[thm]{Lemma}
\newtheorem{prop}[thm]{Proposition}

\newtheorem{defi}[thm]{Definition}

\newtheorem{rek}[thm]{Remark}

\newtheorem{cla}[thm]{Claim}

\numberwithin{equation}{section}

\begin{document}

\title{A Positive Operator-Valued Measure for an Iterated Function System}

\maketitle

\begin{center}Trubee Davison\footnote{Department of Mathematics \\ \indent University of Colorado \\ \indent Campus Box 395 \\ \indent Boulder, CO 80309 \\ \indent trubee.davison@colorado.edu }\end{center}

\begin{abstract} Given an iterated function system (IFS) on a complete and separable metric space $Y$, there exists a unique compact subset $X \subseteq Y$ satisfying a fixed point relation with respect to the IFS. This subset is called the attractor set, or fractal set, associated to the IFS. The attractor set supports a specific Borel probability measure, called the Hutchinson measure, which itself satisfies a fixed point relation. P. Jorgensen generalized the Hutchinson measure to a projection-valued measure, under the assumption that the IFS does not have essential overlap \cite{Jorgensen2} \cite{Jorgensen}. In previous work, we developed an alternative approach to proving the existence of this projection-valued measure \cite{Davison} \cite{Davison2} \cite{Davison3}. The situation when the IFS exhibits essential overlap has been studied by Jorgensen and colleagues in \cite{Kornelson}. We build off their work to generalize the Hutchinson measure to a positive-operator valued measure for an arbitrary IFS, that may exhibit essential overlap. This work hinges on using a generalized Kantorovich metric to define a distance between positive operator-valued measures. It is noteworthy to mention that this generalized metric, which we use in our previous work as well, was also introduced by R.F. Werner to study the position and momentum observables, which are central objects of study in the area of quantum theory \cite{Werner}. We conclude with a discussion of Naimark's dilation theorem with respect to this positive operator-valued measure, and at the beginning of the paper, we prove a metric space completion result regarding the classical Kantorovich metric.

\end{abstract}

\tableofcontents

\noindent \keywords[Keywords: Kantorovich metric, Operator-valued measure, Cuntz algebra, Fixed point]
\noindent [Mathematics subject classification 2010: 46C99 - 46L05] \\

\noindent [Publication note: The final publication is available at Springer via \\
http://dx.doi.org/10.1007/s10440-018-0161-6. This version has been updated with small changes to match the final publication.]


\section{Background:} 

In this opening section, we will provide relevant background information, provide an overview of what is to come, and highlight two applications to quantum theory. To begin, let $(Y,d)$ be a complete and separable metric space. 

\begin{defi} A Lipschitz contraction on $Y$ is a map $L: Y \rightarrow Y$ such that 
$$d(L(x), L(y)) \leq r d(x,y)$$ 

\noindent for all $x,y, \in Y$, where $0 < r < 1$.   

\end{defi}

Let $L: Y \rightarrow Y$ be a Lipschitz contraction on $Y$. Since $Y$ is a complete metric space, it is well known that $L$ admits a unique fixed point $y \in Y$, meaning that $L(y) = y$.  This result is known as the Contraction Mapping Principle, or the Banach Fixed Point Theorem.  In 1981, J. Hutchinson published a seminal paper (see \cite{Hutchinson}), where he generalized the Contraction Mapping Principle to a finite family, $\mathcal{S} = \{\sigma_0,...,\sigma_{N-1}\}$, of Lipschitz contractions on $Y$, where $N \in \mathbb{N}$ is such that $N \geq 2$. Indeed, one can associate to $\mathcal{S}$ a unique compact subset $X \subseteq Y$ which is invariant under the $\mathcal{S}$, meaning that 

\begin{equation} \label{fractal} X = \bigcup_{i=0}^{N-1} \sigma_i(X). \end{equation} 

A finite family of Lipschitz contractions on $Y$ is called an iterated function system (IFS) on $Y$, and the compact invariant subset $X$ described above is called the self-similar fractal set, or attractor set, associated to the IFS.  The existence and uniqueness of the attractor set can be obtained in several ways. One way is the following: if $A, B \subseteq Y$, the Hausdorff distance, $\delta$, between $A$ and $B$ is defined by

$$\delta(A,B)=\sup\{d(a,B),d(b,A):a \in A, b \in B\}.$$

\noindent Denote by $\mathcal{K}$ the collection of compact subsets of $Y$.   It is well known that the metric space $(\mathcal{K}, \delta)$ is complete.  The following theorem guarantees the existence and uniqueness of $X$. 

\begin{thm}\cite{Hutchinson}\cite{Barnsley}[Hutchinson, Barnsley] The Hutchinson-Barnsley operator $F: \mathcal{K} \rightarrow \mathcal{K}$ given by $$K \mapsto  \bigcup_{i=0}^{N-1} \sigma_i(K)$$ is a Lipschitz contraction on the complete metric space $(\mathcal{K}, \delta)$.  
\noindent By the Contraction Mapping Principle, there exists a unique compact subset $X \subseteq Y$ such that $F(X) = X,$ which is equation (\ref{fractal}). 
\end{thm}

In this paper, we will consider the attractor set from a measure theoretic perspective. In particular, the attractor set can be realized as the support of a Borel probability measure on $Y$. This measure, which we denote by $\mu$, satisfies the fixed point relation

\begin{equation} \label{hutmeasure} \mu(\cdot) = \sum_{i=0}^{N-1} \frac{1}{N}\mu(\sigma_i^{-1}(\cdot)), \end{equation}

\noindent and is often referred to as the Hutchinson measure. It is the unique fixed point of a Lipschitz contraction, $T$, on an appropriate complete metric space of Borel probability measures on $Y$. Naturally, the map $T$ is given by 

$$T(\nu) = \sum_{i=0}^{N-1} \frac{1}{N}\nu(\sigma_i^{-1}(\cdot)),$$ 

\noindent for $\nu$ a Borel probability measure on $Y$. The metric, $H$, is given by 

$$H(\mu,\nu) = \sup_{f \in \text{Lip}_1(Y)} \left\{ \left| \int_Y{f}d\mu - \int_Y{f}d\nu \right| \right\},$$ 

\noindent where $\text{Lip}_1(Y) = \{f:Y \rightarrow \mathbb{R}: |f(x)-f(y)| \leq d(x,y) \text{ for all } x,y \in Y\}$, and where $\mu$ and $\nu$ are Borel probability measures on $Y$. This metric is called the Kantorovich metric. 

What is the appropriate complete metric space of Borel probability measures on Y? To answer this question, we make the following definitions: 

\begin{itemize} 

\item Let $Q(Y)$ be the collection of all Borel probability measures on $Y$. 

\item Let $M_{\text{loc}}(Y)$ be the collection of Borel probability measures on $Y$ that have bounded support. 

\item Let $M(Y)$ be the collection of Borel probability measures $\nu$ on $Y$ such that $\int |f| d\nu < \infty$ for all $f \in \text{Lip}(Y)$, where $\text{Lip}(Y)$ is the collection of all real-valued Lipschitz functions on $Y$. 

\end{itemize} 

\noindent We consider two cases: 

\begin{enumerate} 

\item In the case that $(Y,d)$ is a compact (and therefore bounded) metric space, $Q(Y) = M_{\text{loc}}(Y) = M(Y)$. Moreover, in this case, the Kantorovich metric is well-defined (finite) on $Q(Y)$. Indeed, it is well known that $(Q(Y), H)$ is a compact metric space (and therefore complete). It is henceforth appropriate to consider the Contraction Mapping Principle on $(Q(Y), H)$, with respect to the Lipschitz contraction $T$. \\

\item In the case that $(Y,d)$ is an arbitrary complete and separable metric space, the Kantorovich metric is not necessarily well-defined (finite) on $Q(Y)$. Accordingly, we must restrict the Kantorovich metric to a sub-collection of Borel probability measures on $Y$, where it is well-defined. The intent is to find a sub-collection of measures such that the resulting metric space is complete, and such that $T$ restricts to a map on this sub-collection. To our knowledge, there are two choices for a sub-collection which have been studied in the literature. Hutchinson suggested (see \cite{Hutchinson}) that $(M_{\text{loc}}(Y),H)$ constitutes a complete metric space. In a later paper, A. Kravchenko indicated this not to be true, and showed that $(M(Y), H)$ is a complete metric space (see \cite{Kravchenko}). 

\end{enumerate}

The following theorem assures the existence and uniqueness of a measure satisfying equation (\ref{hutmeasure}). It can be shown that the support of this measure is the attractor set $X$. 

\begin{thm} \cite{Kravchenko} [Kravchenko] The map $T: M(Y) \rightarrow M(Y)$ given by 

$$T(\nu) = \sum_{i=0}^{N-1} \frac{1}{N}\nu(\sigma_i^{-1}(\cdot)),$$

\noindent is a Lipschitz contraction on the complete metric space $(M(Y), H)$. By the Contraction Mapping Principle, there exists a unique Borel probability measure $\mu \in M(Y)$ such that $T(\mu) = \mu$, which is equation (\ref{hutmeasure}). 
\end{thm} 

Although $M_{\text{loc}}(Y)$ does not constitute a complete metric space in the $H$ metric, it is interesting to investigate its relationship to $M(Y)$ with respect to the $H$ metric. Indeed, the first result of the present paper will be to show that the metric space completion of $M_{\text{loc}}(Y)$ is $M(Y)$. We will follow this result by examining a bounded version of the Kantorovich metric, which we denote by $MH$, and is given by

$$MH(\mu, \nu) = \sup \left\{ \left| \int_Y f d\mu - \int_Y f d\nu \right| : f \in \text{Lip}_1(Y) \text{ and } ||f||_{\infty} \leq 1 \right \}.$$

\noindent By conglomerating a result of Kravchenko with a previously known result about the weak topology on $Q(Y)$, we will be able to observe that $(Q(Y), MH)$ is a complete metric space. Moreover, we will show that $(Q(Y), MH)$ is the metric space completion of $(M_{\text{loc}}(Y),MH)$. 

\begin{rek} The advantage of the $MH$ metric is that it is well defined on $Q(Y)$; that is, there is no need to restrict the metric to a sub-collection of Borel probability measures on $Y$. However, the reason that the $MH$ metric is not used in the study of iterated function systems is that the condition $||f||_{\infty} \leq 1$ prevents $T$ from being a Lipschitz contraction in the $MH$ metric. 

\end{rek}

We now proceed to the functional analytic setting, with the goal of discussing a generalization of the Hutchinson measure to an operator-valued measure. Consider the Hilbert space $L^2(X, \mu)$, where $X \subseteq Y$ is the attractor set of the IFS $\mathcal{S} = \{ \sigma_0, ..., \sigma_{N-1}\}$, and $\mu$ is the unique Borel probability measure on $Y$ satisfying equation (\ref{hutmeasure}), whose support is $X$. We first consider the case that equation (\ref{fractal}) is a disjoint union. We further assume that there exists a Borel measurable function $\sigma: X \rightarrow X$ such that $\sigma \circ \sigma_i = \text{id}_X$, for all $0 \leq i \leq N-1$. We provide a standard example for the above scenario:  

\begin{itemize} 

\item Let $X = \text{Cantor Set} \subseteq [0,1]$, with the standard metric on $\mathbb{R}$.

\item Let $\sigma_0(x) = \frac{1}{3}x$ and $\sigma_1(x) = \frac{1}{3}x + \frac{2}{3}$. 

\item Let $\sigma(x) = 3x \text{ mod } 1$. 

\end{itemize} 

\noindent Under these assumptions, define

\begin{center} $S_i: L^2(X,\mu) \rightarrow L^2(X,\mu)$ by $\displaystyle{\phi \mapsto (\phi \circ \sigma) \sqrt{N} {\bf{1}}_{{\sigma_i(X)}}}$ \end{center} 

\noindent for all $i = 0, ..., N-1$, and its adjoint

\smallskip

\begin{center} $S_i^*: L^2(X,\mu) \rightarrow L^2(X, \mu)$ by $\displaystyle{\phi \mapsto \frac{1}{\sqrt{N}} (\phi \circ \sigma_i)}$ \end{center} 

\noindent for all $i = 0, ..., N-1$.  This leads to the following result due to P. Jorgensen.  

\begin{thm} \cite{Jorgensen3} [Jorgensen] \label{S} The maps $\{S_i: 0 \leq i \leq N-1\}$ are isometries, and the maps $\{S_i^* : 0 \leq i \leq N-1\}$ are their adjoints.  Moreover, these maps and their adjoints satisfy the Cuntz relations:

\begin{enumerate} 

\item $\displaystyle{\sum_{i=0}^{N-1} S_iS_i^* =  {\bf{1}}_{\mathcal{H}}}$ \label{Cuntz1}

\item $\displaystyle{S_i^*S_j = \delta_{i,j}{\bf{1}}_{\mathcal{H}}}$ \label{Cuntz2} where $0 \leq i,j \leq N-1$.

\end{enumerate}

\end{thm} 

\begin{rek} Another way to rephrase the above theorem is to say that the Hilbert space $L^2(X, \mu)$ admits a representation of the Cuntz algebra, $\mathcal{O}_N$, on $N$ generators. 

\end{rek}

Let $\Gamma_N = \{0,..., N-1\}$.  For $k \in \mathbb{Z}_{+}$, let $\Gamma_N^k = \Gamma_N \times ... \times  \Gamma_N$, where the product is $k$ times.  If $a = (a_1,...,a_k) \in \Gamma_N^k$, where $a_j \in \{0,1,..., N-1\}$ for $1 \leq j \leq k$, define 

\begin{center} $A_k(a) = \sigma_{a_1} \circ ... \circ \sigma_{a_k} (X)$. \end{center} 

\noindent Using that equation (\ref{fractal}) is a disjoint union, we conclude that $\{A_k(a)\}_{a \in \Gamma_N^k}$ partitions $X$ for all $k \in \mathbb{Z}_+$.  For $k \in \mathbb{Z}_+$ and $a = (a_1,..., a_k) \in \Gamma_N^k$ define

\begin{center} $P_k(a) = S_aS_a^*$, \end{center}

\noindent where $S_a = S_{a_1} \circ ... \circ S_{a_k}$. 

\begin{rek} \label{mult} One can show that $P_k(a) = M_{ {\bf{1}}_{A_k(a)}},$ where $M_{ {\bf{1}}_{A_k(a)}}: L^2(X, \mu) \rightarrow L^2(X, \mu)$ is given by $f \in L^2(X,\mu) \mapsto {\bf{1}}_{A_k(a)} f.$ 
\end{rek}

We now recall two definitions. 

\begin{defi} Let $\mathcal{H}$ be a Hilbert space. If $\mathcal{B}(\mathcal{H})$ denotes the $C^*$-algebra of bounded operators on $\mathcal{H}$, a projection $P \in \mathcal{B}(\mathcal{H})$ satisfies $P^* = P$ (self-adjoint) and $P^2 = P$ (idempotent).

\end{defi}

In view of Remark \ref{mult}, note that $P_k(a)$ is a projection in $\mathcal{B}(L^2(X, \mu))$. For the next definition, denote the empty set by $\emptyset$. 

\begin{defi} Let $(X, \mathscr{B}(X))$ be a measure space, and let $\mathcal{H}$ be a Hilbert space. A projection-valued measure with respect to the pair $(X,\mathcal{H})$ is a map $F: \mathscr{B}(X) \rightarrow \mathcal{B}(\mathcal{H})$ such that:

\begin{itemize} 

\item $F(\Delta)$ is a projection in $\mathcal{B}(\mathcal{H})$ for all $\Delta \in \mathscr{B}(X)$;

\item $F(\emptyset) = 0$ and $F(X) = \text{id}_{\mathcal{H}}$ (the identity operator on $\mathcal{H}$);

\item $F(\Delta_1 \cap \Delta_2) = F(\Delta_1)F(\Delta_2)$ for all $\Delta_1, \Delta_2 \in \mathscr{B}(X)$ (where the product operation $F(\Delta_1)F(\Delta_2)$ is operator composition in $\mathcal{B}(\mathcal{H}))$;

\item If $\{ \Delta_n \}_{n=1}^{\infty}$ is a sequence of pairwise disjoint sets in $\mathscr{B}(X)$, and if $g,h \in \mathcal{H}$, then
$$ \left \langle F\left( \bigcup_{n=1}^{\infty} \Delta_n \right)g , h \right \rangle = \sum_{n=1}^{\infty} \langle F(\Delta_n)g, h \rangle.$$

\end{itemize}

\end{defi}

With this definition, we have the following result of Jorgensen. 

\begin{thm} \cite{Jorgensen2} \cite{Jorgensen} [Jorgensen] \label{uniquepvm} There exists a unique projection-valued measure, $E(\cdot)$, defined on the Borel subsets of $X$, $\mathscr{B}(X)$, taking values in the projections on $L^2(X, \mu)$ such that, 

\begin{enumerate} 

\item $E(\cdot) = \sum_{i=0}^{N-1} S_i E(\sigma_i^{-1}(\cdot)) S_i^*$, and 

\item $E(A_k(a)) = P_k(a)$ for all $k \in \mathbb{Z}_+$ and $a \in \Gamma_N^k$.  

\end{enumerate}

\end{thm}

\begin{rek} The projection-valued measure $E(\cdot)$ is the canonical projection-valued measure on the measure space $(X, \mathscr{B}(X), \mu),$ meaning that if $\Delta \in \mathscr{B}(X)$, 

$$E(\Delta) = M_{{\bf{1}}_{\Delta}},$$

\noindent where $M_{{\bf{1}}_{\Delta}}$ is multiplication by ${\bf{1}}_{\Delta}$. 

\end{rek}

Previously, we presented an alternative approach to proving the above theorem (see\cite{Davison} \cite{Davison2} \cite{Davison3}). This is summarized below. Let $P(X)$ be the collection of all projection-valued measures from $\mathscr{B}(X)$ into the projections on $L^2(X, \mu)$. Define a metric $\rho$ on $P(X)$ by 

\begin{equation} \label{genKant} \rho(E,F) = \sup_{f \in \text{Lip}_1(X)}\left\{ \left|\left| \int f dE - \int f dF \right|\right| \right\}, \end{equation}

\noindent where $||\cdot||$ denotes the operator norm in $\mathcal{B}(\mathcal{H})$, and $E$ and $F$ are arbitrary members of $P(X)$.  This is called the generalized Kantorovich metric. 

At this juncture, we would like to note that the generalized Kantorovich metric defined above in equation (see (\ref{genKant})) has been previously defined by R.F. Werner in the setting of mathematical physics, namely in the area of quantum theory (see \cite{Werner}). Indeed, a projection-valued measures is a more specific instance of a positive operator-valued measure (POVM), which is also called an observable in physics. Werner introduces the generalized Kantorovich metric as a tool for studying the position and momentum observables, which are central objects of study in quantum theory. 

In the present paper, and in our related paper (\cite{Davison}), we develop new properties of the generalized Kantorovich metric, and discuss its application to iterated function systems.

\begin{thm} \cite{Davison} \cite{Davison3} [Davison] \label{complete} $(P(X), \rho)$ is a complete metric space. \end{thm} 

\begin{thm} \cite{Davison} \cite{Davison2} \cite{Davison3} [Davison] \label{U} The map $U: P(X) \rightarrow P(X)$ given by

$$F(\cdot) \mapsto \sum_{i=0}^{N-1} S_i F(\sigma_i^{-1}(\cdot)) S_i^*$$

\noindent is a Lipschitz contraction on the $(P(X), \rho)$ metric space. By the Contraction Mapping Principle, there exists a unique projection-valued measure, $E \in P(X),$ satisfying part $(1)$ of Theorem \ref{uniquepvm}.  Part $(2)$ of Theorem \ref{uniquepvm} follows as a consequence. 

\end{thm} 

The fact that $U(F) \in P(X)$ for $F \in P(X)$ depends on the Cuntz relations. For instance, the computation that $U(F)$ is an idempotent relies on part (2) of Theorem \ref{S}. Indeed, if $\Delta \in \mathscr{B}(X)$

\begin{eqnarray*} 
(U(F)(\Delta))^2 & = & \left(\sum_{i=0}^{N-1} S_i F(\sigma_i^{-1}(\Delta))S_i^*\right)^2 \\
			      & = & \sum_{i=0}^{N-1} S_i F(\sigma_i^{-1}(\Delta))S_i^* \sum_{j=0}^{N-1} S_j F(\sigma_j^{-1}(\Delta))S_j^* \\
			      & = & \sum_{i=0}^{N-1} S_i F(\sigma_i^{-1}(\Delta))^2S_i^* \\
			      & = & \sum_{i=0}^{N-1} S_i F(\sigma_i^{-1}(\Delta))S_i^* \\
			      & =& U(F)(\Delta). \end{eqnarray*} 
			     
\noindent More generally, the fact that $U(F)(\Delta_1 \cap \Delta_2) = U(F)(\Delta_1)U(F)(\Delta_2)$ similarly relies on part (2) of Theorem \ref{S}. The fact that $U(F)(X) = \text{id}_{\mathcal{H}}$ relies on part (1) of Theorem \ref{S}. That is,

\begin{eqnarray*} 
U(F)(X) & = & \sum_{i=0}^{N-1} S_i F(\sigma_i^{-1}(X))S_i^* \\
			      & = & \sum_{i=0}^{N-1} S_i F(X)S_i^* \\
			      & = & \sum_{i=0}^{N-1} S_i \text{id}_{\mathcal{H}} S_i^* \\
			      & = & \sum_{i=0}^{N-1} S_iS_i^* \\
			      & =& \text{id}_{\mathcal{H}}. \end{eqnarray*} 
			      
\begin{rek} \label{POVM} We see by the above computations that part (2) of Theorem \ref{S} is used to show that $U(F)$ takes values in the projections on $L^2(X, \mu)$. However, if we hypothetically obtain a family of operators $\{S_i\}_{i=0}^{N-1}$ satisfying part (1) of Theorem \ref{S}, and not part (2) of Theorem \ref{S}, the map $U$ will still carry structure. In particular, if $F$ is a POVM, $U(F)$ will also be a POVM. This situation appears when our iterated function system has essential overlap, which we describe below. 
\end{rek}

The construction of the isometries $\{S_i\}_{i=0}^{N-1}$ that satisfy parts (1) and (2) of Theorem \ref{S} depends on the fact that equation (\ref{fractal}) is a disjoint union, and on the existence of a map $\sigma: X \rightarrow X$ satisfying $\sigma \circ \sigma_i = \text{id}_X$ for all $0 \leq i \leq N-1.$ An IFS exhibiting a disjoint union in equation (\ref{fractal}) is an example of a broader class of iterated function systems called iterated function systems with non-essential overlap, meaning that

\begin{equation} \label{noness} \mu(\sigma_i(X) \cap \sigma_j(X)) = 0, \end{equation}

\noindent when $i \neq j$, and where $\mu$ is the Hutchinson measure. Like an iterated function system which is disjoint, an iterated function system with non-essential overlap also admits a representation of the Cuntz algebra, assuming that each member of the IFS is of finite type. This is due to Jorgensen and his collaborators K. Kornelson, and K. Shuman, and is stated below \cite{Kornelson}. 

\begin{defi} \cite{Kornelson} A measurable endomorphism $\tau: X \rightarrow X$ is said to be of finite type if there is a finite partition $E_1, ..., E_k$ of $\tau(X),$ and measurable mappings $\sigma_i: E_i \rightarrow X,$ $i =1,...,k$ such that 

$$\sigma_i \circ \tau |_{E_i} = \text{id}_{E_i}.$$ 

\end{defi} 

\begin{thm} \cite{Kornelson} [Jorgensen, Kornelson, Shuman] Let $\mathcal{S} = \{\sigma_0,...,\sigma_{N-1}\}$ be an IFS with attractor set $X$, and let $\mu$ be the corresponding Hutchinson measure. Further suppose that each member of the IFS is of finite type. For each $i = 0, ..., N-1,$ define $F_i: L^2(X,\mu) \rightarrow L^2(X, \mu)$ given by 

$$\displaystyle{\phi \mapsto \frac{1}{\sqrt{N}} (\phi \circ \sigma_i)}.$$  

\noindent The family of operators $\{S_i := F_i^*\}_{i=0}^{N-1}$ define a representation of the Cuntz algebra if and only if the IFS has non-essential overlap. 

\end{thm} 

In view of this result, an additional question to ask is what can be said about an IFS , $\mathcal{S} = \{\sigma_0,...,\sigma_{N-1}\}$, that has essential overlap, meaning that equation $\mu(\sigma_i(X) \cap \sigma_j(X)) > 0,$ for some $i \neq j$. More generally, what can be said about an arbitrary IFS, that may or may not exhibit non-essential overlap? This situation was also studied by Jorgensen and his collaborators in \cite{Kornelson}. In particular, we can still define the operators

\begin{center} $F_i: L^2(X,\mu) \rightarrow L^2(X, \mu)$ given by $\displaystyle{\phi \mapsto \frac{1}{\sqrt{N}} (\phi \circ \sigma_i)}$ \end{center} 

\noindent for all $i = 0, ..., N-1$.

\begin{thm} \cite{Kornelson} \label{Kornelson} [Jorgensen, Kornelson, Shuman] The family of operators $\{F_i\}_{i=0}^{N-1}$ satisfy the operator identity 

$$\sum_{i=0}^{N-1} F_i^*F_i = \text{id}_{\mathcal{H}}. $$

\end{thm}

\begin{proof} We include a proof of this result because this result is fundamental to the main purpose of the present paper. Note that the the Hutchinson measure $\mu$ associated to this IFS has the property that 

$$\frac{1}{N} \sum_{i=0}^{N-1} \int_X |f|^2 \circ \sigma_i d\mu = \int_X |f|^2 d\mu = ||f||^2$$

\noindent for all $f \in L^2(X,\mu).$ Since the operator $\sum_{i=0}^{N-1} F_i^*F_i$ is self-adjoint, in order to show that $\sum_{i=0}^{N-1} F_i^*F_i = \text{id}_{\mathcal{H}},$ it is enough to show that 

$$\left \langle \left( \sum_{i=0}^{N-1} F_i^*F_i \right) f, f \right \rangle = \langle f, f \rangle = ||f||^2,$$

\noindent for all $f \in L^2(X,\mu).$ Accordingly, let $f \in L^2(X, \mu),$ and observe that for all $0 \leq i \leq N-1,$

$$||F_if||^2 = \langle F_i f, F_i f \rangle = \int_X \frac{1}{\sqrt{N}} (f \circ \sigma_i) \overline{\frac{1}{\sqrt{N}} (f \circ \sigma_i)} d\mu = \frac{1}{N} \int_X |f|^2 \circ \sigma_i d\mu.$$

\noindent Therefore, 

$$\left \langle \left( \sum_{i=0}^{N-1} F_i^*F_i \right) f, f \right \rangle = \sum_{i=0}^{N-1} \langle F_i^*F_i f, f, \rangle = \sum_{i=0}^{N-1} ||F_i f ||^2$$

$$ = \frac{1}{N} \sum_{i=0}^{N-1} \int_X |f|^2 \circ \sigma_i d\mu = ||f||^2,$$

\noindent which proves the result. 

\end{proof} 

Consequently, if we define $S_i = F_i^*$, we can rewrite the above operator identity as 

\begin{equation} \label{ident} \sum_{i=0}^{N-1} S_iS_i^* = \text{id}_{\mathcal{H}}, \end{equation} 

\noindent which is exactly part (1) of Theorem \ref{S}.

Referring back to Remark \ref{POVM}, our intent is to generalize the map $U$ from Theorem \ref{U} to the case of an arbitrary IFS; that is, to the case that we have a family of operators $\{S_i\}_{i=0}^{N-1}$ that satisfy part (1) of Theorem \ref{S}. Toward this end, let $S(X)$ be the collection of all positive operator-valued measures from $\mathscr{B}(X)$ into the positive operators on $L^2(X, \mu)$. We will show that $(S(X), \rho)$ is a complete metric space, thereby generalizing Theorem \ref{complete} to positive operator-valued measures. Additionally, we will show that the map $U$ extends to a map on $S(X)$, and is a Lipschitz contraction in the $\rho$ metric. As a consequence, there will exist a unique POVM $A \in S(X)$ satisfying

\begin{equation} \label{uniquepovm} A(\cdot) = \sum_{i=0}^{N-1} S_i A(\sigma_i^{-1}(\cdot)) S_i^*. \end{equation}

\noindent This will generalize part (1) of Theorem \ref{uniquepvm} to the case of an arbitrary IFS. We would like to note that Jorgensen proved the existence and uniqueness of a POVM that satisfies equation (\ref{uniquepovm}) in a special case, using Kolmogorov's extension theorem. We refer the reader to Lemma I.3 in \cite{Jorgensen4}. 

A family of operators $\{S_i\}_{i=0}^{N-1}$ defined on a Hilbert space satisfying equation (\ref{ident}) is one of  the starting points for our results below. To provide broader context, it is worthwhile to note that equation (\ref{ident}) also appears in quantum information theory. Indeed, a measurement of a quantum system on a Hilbert space $\mathcal{H}$ is composed of a family of such operators, as described in papers by D.W. Kribs and colleagues (see \cite{Kribs} and \cite{Kribs2}). The number of operators in the family corresponds to the number of measurement outcomes of an experiment. If the state of the system before the experiment is the unit vector $h \in \mathcal{H}$, then the probability that measurement outcome $i$ occurs is given by $p_h(i) = \langle S_iS_i^* h, h \rangle$. Using equation (\ref{ident}), we obtain $\sum_{i=0}^{N-1} p_h(i) = 1$, which implies that $p_h(\cdot)$ is a probability measure on the measurement outcomes. This connection to quantum information theory was made aware to the author by Jorgensen and colleagues in their paper \cite{Kornelson}. 

To conclude the introductory section, we mention an important result in operator theory, which is Naimark's dilation theorem. 

\begin{thm} \text{} [Naimark's Dilation Theorem] Let $F$ be a POVM with respect to the pair $(X, \mathcal{H})$. There exists a Hilbert space $\mathcal{K}$, a bounded operator $V: \mathcal{K} \rightarrow \mathcal{H},$ and a projection-valued measure $P$ with respect to the pair $(X, \mathcal{K}),$ such that

$$F(\cdot) = VP(\cdot)V^*.$$

\end{thm} 

\noindent We call $P$ a dilation of the POVM $F$. In the results section, we will build off an existing result in \cite{Kornelson} to identify an explicit Hilbert space that supports such a dilation of the POVM $A$ in equation (\ref{uniquepovm}).

\section{Results:}

\subsection{Metric Space Completion of $M_{\text{loc}}(Y)$:}

Let $(Y,d)$ be a complete and separable metric space. As mentioned earlier, it was first claimed in \cite{Hutchinson}  that $(M_{\text{loc}}(Y),H)$ is a complete metric space.  However, we will briefly outline an example, presented in \cite{Kravchenko}, which shows this not to be true.   

\begin{cla}\cite{Kravchenko}[Kravchenko] Let $(Y,d)$ be an unbounded metric space.  Then \\ $(M_{\text{loc}}(Y),H)$ is not complete.  \end{cla}

\begin{proof} Choose a sequence of points $x_k \in Y$ for $k=0,1,2,...$, such that $d(x_0,x_k) \leq k $ for all $k$, and $d(x_k,x_0) \rightarrow \infty$.  For a point $x \in Y$, define the delta measure at $x$ by 
$$
\delta_x(A) =
\begin{cases}
1 \text{ if } x \in A \\
0 \text{ if } x \notin  A.
\end{cases}
$$

\noindent For $n=1,2,3,...$, define the sequence of measures $\nu_n = 2^{-n}\delta_{x_0} + \Sigma_{k=1}^n 2^{-k}\delta_{x_k} \in M_{\text{loc}}(Y)$. This sequence is Cauchy in $(M_{\text{loc}}(Y),H)$.  However, it can be shown that it does not converge to a measure in $(M_{\text{loc}}(Y),H)$.

\end{proof} 

Since $(M_{\text{loc}}(Y),H)$ is not a complete metric space  (when $Y$ is unbounded), we consider the larger sub-collection of measures, $M(Y)$, equipped with the $H$ metric.  Indeed, we will review that $M_{\text{loc}}(Y) \subseteq M(Y)$. 

\begin{defi} A measure $\mu$ on the metric space $Y$ is said to be regular if for every Borel subset $A \subseteq Y$, and every $\epsilon > 0$, there exists a closed set $F$ and an open set $G$ such that $F \subseteq A \subseteq G$ and $\mu(G \setminus F) < \epsilon$.  
\end{defi}

\begin{defi} A measure $\mu$ on the metric space $Y$ is said to be tight if for every $\epsilon > 0$, there exists a compact set $K$ such that $\mu(Y \setminus K) < \epsilon$.  
\end{defi} 

\begin{rek} \label{regular} Since $Y$ is a complete and separable metric space, every Borel probability measure on $Y$ is regular and tight (see Ch. 1, Section 1 in \cite{Billingsley}).  In particular, the measures in $M(Y)$ and  $M_{\text{loc}}(Y)$ are all regular and tight. 
\end{rek}

\begin{lem} {\label{supcompact}} \cite{Billingsley}[Ch. 1, Section 1 in Billingsley] A Borel probability measure $\mu$ is tight on the metric space $Y$ if and only if for each Borel subset $A \subseteq Y$, $\mu(A) = \sup\{\mu(K): K \subseteq A \text{ and } K \text{ compact }\}$.
\end{lem}

\begin{cor} \label{supportzero} If $\mu$ is a Borel probability measure which is tight on the metric space $Y$, then $\mu(Y \setminus \text{supp}(\mu)) = 0$.  
\end{cor} 

\begin{proof} Note that $Y \setminus \text{supp}(\mu) = \cup\{A \subseteq Y: A \text{ is open and } \mu(A)=0\}$ which is a Borel set in $Y$.  Therefore by Lemma \ref{supcompact}, $\mu(Y \setminus \text{supp}(\mu)) = \sup\{\mu(K): K \subseteq Y \setminus \text{supp}(\mu) \text{ and } K \text{ compact }\}$.  Now if $K \subseteq Y \setminus \text{supp}(\mu)$, then since $K$ is compact, it has a finite subcovering by $\mu$-measure zero open sets.  Hence, $\mu(K) = 0$, and therefore $\mu(Y \setminus \text{supp}(\mu)) = 0$.
\end{proof} 

\begin{prop} $M_{\text{loc}}(Y) \subseteq M(Y)$.  

\end{prop} 

\begin{proof} Let $\mu \in M_{\text{loc}}(Y)$.  To show that $\mu \in M(Y)$, we need to show that $\int_Y |f| d\mu < \infty$ for all $f \in \text{Lip}(Y)$.  Choose $f \in \text{Lip}(Y)$ with Lipschitz constant $\gamma$, and choose a point $x_0 \in Y$.  Since $\mu$ has bounded support, we can assume that there exists a $K \geq 0$ such that $\text{supp}(\mu) \subseteq B_K(x_0)$, where $B_K(x_0) = \{x \in Y : d(x,x_0) \leq K\}$.  Moreover, $\mu(Y \setminus B_K(x_0)) = 0$ by Corollary \ref{supportzero}.  This implies that
\begin{eqnarray*} 
\int_Y |f| d\mu & = & \int_{B_K(x_0)} |f| d\mu + \int_{Y \setminus B_K(x_0)} |f| d\mu \nonumber \\
		  & = & \int_{B_K(x_0)} |f| d\mu. \nonumber \\ \end{eqnarray*}
		  
\noindent Continuing, observe that
\begin{eqnarray*}  
\int_{B_K(x_0)} |f(x)| d\mu(x) & \leq &  \int_{B_K(x_0)} |f(x) - f(x_0)| d\mu(x) + \int_{B_K(x_0)} |f(x_0)| d\mu(x) \nonumber \\
           & \leq & \int_{B_K(x_0)} \gamma d(x,x_0) d\mu(x) + \int_{B_K(x_0)} |f(x_0)| d\mu(x) \nonumber \\
           & \leq &  \gamma K \mu(B_K(x_0)) + |f(x_0)| \mu(B_K(x_0)) \nonumber \\
           & < &  \infty.  \nonumber 
					 \end{eqnarray*} 
					 
\noindent This shows that $M_{\text{loc}}(Y) \subseteq M(Y)$.  

\end{proof}

As mentioned in the introduction, it was proved by Kravchenko that $(M(Y),H)$ is a complete metric space (see \cite{Kravchenko}). It is worth noting that C. Akerlund-Bistrom proved the special case that if $Y = \mathbb{R}^n$, then $(M(Y),H)$ is a complete metric space (see \cite{Bistrom}). During a seminar talk that the author presented at the University of Colorado, A. Gorokhovsky posed the question: is $(M(Y),H)$ the metric space completion of $(M_{\text{loc}}(Y),H)$? This question is answered in the affirmative below. 

\begin{thm} \cite{Davison3} \label{completion}[Davison] $(M(Y), H)$ is the completion of the metric space \\ $(M_{\text{loc}}(Y),H)$. \end{thm} 

\begin{proof} Suppose that $\mu$ is a Borel probability measure in $M(Y)$.  We need to find a sequence of measures $\{\mu_n\}_{n=1}^{\infty} \subseteq M_{\text{loc}}(Y)$ such that $\mu_n \rightarrow \mu$ in the $H$ metric.  We know from earlier, namely Lemma {\ref{supcompact}}, that there exists a sequence of compact subsets $\{K_n\}_{n=1}^{\infty}$ of $Y$ such that $\lim_{n \rightarrow \infty} \mu(K_n) = 1$.  We can choose this sequence of compact sets such that $K_1 \subseteq K_2 \subseteq K_3 \subseteq ...$, because the union of finitely many compact sets is compact, and because measures are monotone.  Next, choose some $x_0 \in K_1$.  Since each $K_n$ is compact, it is bounded so there exists a positive integer $k_n$ such that $K_n \subseteq B_{k_n}(x_0)$, where $B_{k_n}(x_0) = \{x \in Y : d(x,x_0) \leq k_n\}$.  For each $n= 1, ..., \infty$, define a Borel measure $\mu_n$ on $Y$ by $\displaystyle{\mu_n(\Delta) = \frac{\mu(\Delta \cap K_n)}{\mu(K_n)}}$ for all Borel subsets $\Delta \subseteq Y$.  Furthermore for $f \in \text{Lip}(Y),$ note that
$$\int_Y f d\mu_n = \frac{1}{\mu(K_n)} \int_Y {f \textbf{1}_{K_n}} d\mu.$$

\noindent We claim that each $\mu_n$ has bounded support.  Consider the open set $Y \setminus K_n$. 
$$\mu_n(Y \setminus K_n) = \frac{\mu((Y \setminus K_n)  \cap K_n)}{\mu(K_n)} = 0,$$

\noindent and hence the support of $\mu_n$ is contained within the bounded set $K_n$.  Also, observe that
$$\mu_n(Y) = \frac{\mu(Y \cap K_n)}{\mu(K_n)} = \frac{\mu(K_n)}{\mu(K_n)} = 1,$$

\noindent so that $\mu_n$ is a Borel probability measure on $Y$.  We have shown that for all $n = 1, 2, ...$, $\mu_n \in M_{\text{loc}}(Y)$.  It remains to show that $\mu_n \rightarrow \mu$ in the $H$ metric.  For this we  use the alternate formulation for the $H$ metric which is shown in \cite{Bistrom}; namely
$$H(\mu_n, \mu) = \sup_{f \in \text{Lip}_1(x_0)} \left\{ \left| \int_Y f d\mu_n - \int_Y f d\mu \right| \right\},$$ 

\noindent where $\text{Lip}_1(x_0)$ are the $\text{Lip}_1(Y)$ functions which vanish at $x_0$.  Let $\epsilon > 0$.  Choose some $f \in \text{Lip}_1(x_0)$.  Then  

\vspace{5mm}

$\displaystyle{\left| \int_Y f d\mu_n - \int_Y f d\mu \right|  =  \left| \frac{1}{\mu(K_n)} \int_Y {f\textbf{1}_{K_n}} d\mu - \int_Y {f} d\mu \right|}$ 

\vspace{5mm}

$\displaystyle{=  \frac{1}{\mu(K_n)} \left| \int_Y {f\textbf{1}_{K_n} - \mu(K_n)f} d\mu \right|}$ 

\vspace{5mm}

$\displaystyle{\leq  \frac{1}{\mu(K_n)} \left| \int_{K_n} {(f\textbf{1}_{K_n} - \mu(K_n)f)} d\mu \right|  +  \frac{1}{\mu(K_n)} \left| \int_{Y \setminus K_n} {(f\textbf{1}_{K_n} - \mu(K_n)f)} d\mu \right|}$ 

\vspace{5mm}
										
$\displaystyle{\leq \left( \frac{1-\mu(K_n)}{\mu(K_n)}  \int_{K_n} |f| d\mu \right) + \int_{Y \setminus K_n} |f| d\mu}$ 

\vspace{5mm}

$\displaystyle{\leq \left( \frac{1-\mu(K_n)}{\mu(K_n)} \int_{K_n} d(x,x_0) d\mu \right) + \int_{Y \setminus K_n} d(x,x_0) d\mu :=  I(n),}$ 

\vspace{5mm}
																				
\noindent where the last inequality is because $|f(x)| = |f(x)-f(x_0)| \leq d(x,x_0)$. \\
										
Since $\mu \in M(Y)$ and $d(x,x_0) \in \text{Lip}_1(Y) \subseteq \text{Lip}(Y)$ 
\begin{equation*} 0 \leq \int_Y {d(x,x_0)} d\mu := L < \infty. \end{equation*}  

\noindent Because $d(x,x_0)$ is a non-negative function, we note that for all $n$, $0 \leq \int_{K_n} {d(x,x_0)} d\mu \leq L < \infty$ and $0 \leq \int_{Y \setminus K_n} {d(x,x_0)} d\mu \leq L < \infty. $

Since $\lim_{n \rightarrow \infty} \mu(K_n) = \mu(Y) = 1$, and $K_1 \subseteq K_2 \subseteq ...$, observe that
$\textbf{1}_{Y \setminus K_n}d(x,x_0)$ \\ decreases pointwise to $0$ $\mu$-almost everywhere.  By the dominated convergence theorem, 
\begin{equation*} \lim_{n \rightarrow \infty} \int_{Y \setminus K_n} {d(x,x_0)} d\mu = \lim_{n \rightarrow \infty} \int_{Y} {\textbf{1}_{Y \setminus K_n}d(x,x_0)} d\mu =  \int_{Y} \lim_{n \rightarrow \infty}{\textbf{1}_{Y \setminus K_n}d(x,x_0)} d\mu = 0. \end{equation*}

\noindent Also, $\displaystyle{\lim_{n \rightarrow \infty} \left( \frac{1-\mu(K_n)}{\mu(K_n)} \right) = 0.}$  Choose an $N$ such that for $n \geq N$, 
$$\left( \frac{1-\mu(K_n)}{\mu(K_n)} \right)  \leq \frac{\epsilon}{2L},$$ 

\noindent and
$$ \int_{Y} {\textbf{1}_{Y \setminus K_n}d(x,x_0)} d\mu \leq \frac{\epsilon}{2}.$$

For $n \geq N$, $I(n) \leq \frac{\epsilon}{2L}(L) + \frac{\epsilon}{2} = \epsilon$.  Since the choice of $N$ is independent of the choice of $f \in \text{Lip}_1(x_0)$, we can conclude that $H(\mu_n, \mu) \leq I(n) \leq \epsilon$.  Therefore, we have shown that $M(Y)$ is the completion of the metric space $M_{\text{loc}}(Y)$ in the $H$ metric. 

\end{proof} 

As we remarked previously, the Kantorovich metric is not well defined (finite) on all Borel probability measures on $Y$ when $Y$ is unbounded. In the discussion above, this issue is overcome by restricting the Kantorovich metric to a sub-collection of measures. Another option is to consider a modified Kantorovich metric, $MH$, on $Q(Y)$ defined as follows:  For $\mu, \nu \in Q(Y),$

\begin{equation} MH(\mu,\nu) = \sup \left\{ \left| \int_Y{f}d\mu - \int_Y{f}d\nu \right| : f  \in \text{Lip}_1(Y) \text{ and } ||f||_{\infty} \leq 1 \right\}. \end{equation} 

\noindent The condition $||f||_\infty \leq 1$ guarantees that $MH$ will be finite on $Q(Y)$.  Also, observe that we have the containments: 

$$M_{\text{loc}}(Y) \subseteq M(Y) \subseteq Q(Y).$$ 




We can equip $Q(Y)$ with the weak topology.  Indeed, a net of measures $\{\mu_{\lambda}\}_{\lambda \in \Lambda} \subseteq Q(Y)$ converges weakly to a measure $\mu \in Q(Y)$, if for all $f \in C_b(Y)$, $\int_Y f d\mu_{\lambda} \rightarrow \int_Y f d\mu$, where $C_b(Y)$ is the set of all bounded continuous real-valued functions on $Y$.   The following result can be found in Section 8.3 of \cite{Bogachev}.   

\begin{thm} \cite{Bogachev}[Section 8.3 in Bogachev] The weak topology on $Q(Y)$ coincides with the topology induced by the $MH$ metric on $Y$. 
\end{thm} 

We now state a result recently proved by Kravchenko in \cite{Kravchenko} (which was crucial for showing that $(M(Y), H)$ is complete). First, we put $\text{Lip}_b(Y)$ to be the collection of real-valued bounded Lipschitz functions on $Y.$

\begin{prop} \cite{Kravchenko} \label{keylemma}[Kravchenko] Let $\{\mu_n\}_{n=1}^{\infty}$ be a sequence of Borel measures on the complete and separable metric space $Y$ such that $\mu_n(Y) = K < \infty$ for all $n =1,2,...$, and such that for all $f \in \text{Lip}_b(Y)$, the sequence $\{ \int_Y f d\mu_n \}_{n=1}^{\infty}$ of real numbers is Cauchy.  Then there exists a Borel measure $\mu$ on $Y$ such that $\mu(Y) = K$, and such that the sequence $\{\mu_n\}_{n=1}^{\infty}$ converges in the weak topology to $\mu$. 
\end{prop}

We can combine the above two results to gain the following. 

\begin{thm} $(Q(Y), MH)$ is a complete metric space. 

\end{thm} 

\begin{proof} 

If $\{\mu_n\}_{n=1}^{\infty}$ is a Cauchy sequence of measures in $Q(Y)$, one can show that for all $f \in \text{Lip}_b(Y)$, the sequence $\{ \int_Y f d\mu_n \}_{n=1}^{\infty}$ of real numbers is Cauchy.  Therefore, by the above proposition there will exist a Borel probability measure $\mu \in Q(Y)$ such that $\mu_n$ converges to $\mu$ in the weak topology, or equivalently, in the $MH$ metric. 

\end{proof}

We now adapt Theorem \ref{completion} to this setting.  

\begin{thm} \label{completion2} \cite{Davison3} [Davison] The completion of the metric space $(M_{\text{loc}}(Y), MH)$ \\  is $(Q(Y), MH)$.  
\end{thm}

\begin{proof} The proof of this theorem is similar to the earlier proof of Theorem \ref{completion}.  Suppose that $\mu \in Q(Y)$.  We need to find a sequence of measures $\{\mu_n\}_{n=1}^{\infty} \subseteq M_{\text{loc}}(Y)$ such that $\mu_n \rightarrow \mu$ in the $MH$ metric.  Define, exactly as before, a sequence of measures $\{\mu_n\}_{n=1}^{\infty} \subseteq M_{\text{loc}}(Y)$.  In particular, $\mu_n$ satisfies $\int f d\mu_n = \frac{1}{\mu(K_n)} \int_Y {f  \textbf{1}_{K_n}} d\mu$ for all $f \in \text{Lip}(Y).$  Choose $f \in \text{Lip}_1(X)$ such that $||f||_\infty \leq 1$.  Then
\begin{eqnarray*} 
 \left| \int_Y f d\mu_n - \int_Y f d\mu \right| & = &  \left| \int_Y f d\mu_n - \int_Y f d\mu \right| \\
 								& = & \left| \frac{1}{\mu(K_n)} \int_Y {f\textbf{1}_{K_n}} d\mu - \int_Y {f} d\mu \right| \nonumber \\
										& = & \frac{1}{\mu(K_n)} \left| \int_Y {f\textbf{1}_{K_n} - \mu(K_n)f} d\mu \right| \nonumber \\
										& \leq & \frac{1}{\mu(K_n)} \left| \int_{K_n} {(f\textbf{1}_{K_n} - \mu(K_n)f)} d\mu \right| \nonumber \\ & + & \frac{1}{\mu(K_n)} \left| \int_{Y \setminus K_n} {(f\textbf{1}_{K_n} - \mu(K_n)f)} d\mu \right| \nonumber \\
										& \leq & \left( \frac{1-\mu(K_n)}{\mu(K_n)}  \int_{K_n} |f| d\mu \right) + \int_{Y \setminus K_n} |f| d\mu \nonumber \\
										& \leq & \left( \frac{1-\mu(K_n)}{\mu(K_n)} \int_{K_n} 1 d\mu \right) + \int_{Y \setminus K_n} 1 d\mu \nonumber \\
										& \leq & (1-\mu(K_n)) + \mu(Y \setminus K_n) \\
										& = & 2 \mu(Y \setminus K_n). \\
										\end{eqnarray*} 
										
\noindent  The last line of the above expression is independent of the choice of $f$ and goes to zero as $n$ goes to infinity.  Hence, $\mu_n \rightarrow \mu$ in the $MH$ metric.

\end{proof}

\subsection{Generalizing the Kantorovich Metric to Positive Operator-Valued Measures}

In this sub-section, we will generalize the Kantorovich metric to the space of positive operator-valued measures on a Hilbert space, which are operator-valued measures which take values in the positive operators. The positive operators on a Hilbert space contain the projections.   

Let $(X,d)$ be a compact metric space, and let $\mathcal{H}$ be an arbitrary Hilbert space. In our application in the next sub-section, $X$ will be the attractor set associated to an IFS, and $\mathcal{H} = L^2(X, \mu),$ where $\mu$ is the Hutchinson measure. 

We begin with some preliminary definitions and facts. 

\begin{defi} A positive operator $L \in \mathcal{B}(\mathcal{H})$ satisfies $\langle Lh, h \rangle \geq 0$ for all $h \in \mathcal{H}$. 
\end{defi}

\begin{defi} A positive operator-valued measure with respect to the pair $(X,\mathcal{H})$ is a map $A: \mathscr{B}(X) \rightarrow \mathcal{B}(\mathcal{H})$ such that:

\begin{itemize} 

\item $A(\Delta)$ is a positive operator in $\mathcal{B}(\mathcal{H})$ for all $\Delta \in \mathscr{B}(X)$;

\item $A(\emptyset) = 0$ and $A(X) = \text{id}_{\mathcal{H}}$ (the identity operator on $\mathcal{H}$);

\item If $\{ \Delta_n \}_{n=1}^{\infty}$ is a sequence of pairwise disjoint sets in $\mathscr{B}(X)$, and if $g,h \in \mathcal{H}$, then
$$ \left \langle A\left( \bigcup_{n=1}^{\infty} \Delta_n \right)g , h \right \rangle = \sum_{n=1}^{\infty} \langle A(\Delta_n)g, h \rangle.$$

\end{itemize}

\end{defi}

\begin{rek} A projection-valued measure with respect to the pair $(X,\mathcal{H})$ is a positive operator-valued measure because projections are positive operators.  
\end{rek}

\begin{rek} \label{Asesqui} Let $A$ be a positive operator-valued measure with respect to the pair $(X, \mathcal{H})$.  The map $[g,h] \in \mathcal{H} \times \mathcal{H} \mapsto A_{g,h}(\cdot)$ is sesquilinear.  This follows from the fact that the inner product on $\mathcal{H}$ is sesquilinear. 
\end{rek} 

Our below discussion will rely on the following two standard theorems of functional analysis, which are stated with the amount of generality we will need.  

\begin{thm} \cite{Conway} \label{riesz1}[Theorem III.5.7 in Conway] Let $X$ be a compact metric space, and $T: C(X) \rightarrow \mathbb{C}$ be a bounded linear functional.  There exists a unique complex-valued regular Borel finite measure $\mu$ on $X$ such that
$$\int_X f d\mu = T(f),$$ 

\noindent for all $f \in C(X)$, and such that $||\mu|| = ||T||$ (where $||\mu||$ denotes the total variation norm of $\mu$).   

\end{thm}

\begin{thm} \cite{Conway} \label{boundedsesqui}[Theorem II.2.2 in Conway] Let $u: \mathcal{H} \times \mathcal{H} \rightarrow \mathbb{C}$ be a bounded sesquilinear form with bound $M$.  There exists a unique operator $A \in \mathcal{B}(\mathcal{H})$ such that $u(g,h) = \langle Ag, h \rangle$ for all $g,h \in \mathcal{H}$, and such that $||A|| \leq M$.
\end{thm}

Let $S(X)$ be the collection of all positive operator-valued measures with respect to the pair  $(X,\mathcal{H})$. Consider the metric $\rho$ on $S(X)$.  That is, 
\begin{equation} \rho(A,B) = \sup_{f \in \text{Lip}_1(X)}\left\{ \left|\left| \int f dA - \int f dB \right|\right| \right\}, \end{equation}

\noindent where $||\cdot||$ denotes the operator norm in $\mathcal{B}(\mathcal{H})$, and $A$ and $B$ are arbitrary members of $S(X)$.

\begin{thm} \label{SXcomplete} \cite{Davison3} [Davison]  The metric space $(S(X), \rho)$ is complete.  
\end{thm} 

\begin{proof} Let $\{A_n\}_{n=1}^{\infty} \subseteq S(X)$ be a Cauchy sequence in the $\rho$ metric.  We have the following claim. 

\begin{cla} \label{cauchy} Let $f \in C(X)$.  The sequence of operators $\{ A_n(f) := \int f dA_n\}_{n=1}^{\infty}$ is Cauchy in the operator norm.   

\end{cla}

\noindent Proof of claim: Note that the proof of this claim is identical to the proof of Claim 2.12 in \cite{Davison}.  Let $\epsilon > 0$.  Let $f = f_1 + if_2$, where $f_1, f_2 \in C_{\mathbb{R}}(X),$ where $C_{\mathbb{R}}(X)$ is the collection of real-valued continuous functions on $\mathbb{R}.$ Since $X$ is compact, by the density of Lipschitz functions in continuous functions we can choose $g_1, g_2 \in \text{Lip}(X)$ such that $||f_1-g_1||_{\infty} \leq \frac{\epsilon}{6}$ and $||f_2-g_2||_{\infty} \leq \frac{\epsilon}{6}$.  

There is a $K > 0$ such that $\frac{1}{K}g_1 \in \text{Lip}_1(X)$ and $\frac{1}{K}g_2 \in \text{Lip}_1(X)$.  Since $\{A_n\}_{n=1}^{\infty}$ is a Cauchy sequence in the $\rho$ metric, the sequence $\{A_n(\frac{1}{K}g_1)\}_{n=1}^{\infty}$ is Cauchy in the operator norm, and hence the sequence $\{A_n(g_1)\}_{n=1}^{\infty}$ is Cauchy in the operator norm.  Similarly, $\{A_n(g_2)\}_{n=1}^{\infty}$ is Cauchy in the operator norm.  Therefore, choose $N$ such that for $n,m \geq N$, $$||A_n(g_1) - A_m(g_1)|| \leq \frac{\epsilon}{6} \text{ and } ||A_n(g_2) - A_m(g_2)|| \leq \frac{\epsilon}{6}.$$  If $m,n \geq N$,
\begin{eqnarray*}
||A_n(f_1) - A_m(f_1)|| & \leq & ||A_n(f_1) - A_n(g_1) || + ||A_n(g_1) - A_m(g_1)||  \\
			      & + & ||A_m(g_1) - A_m(f_1)|| \\
			      & \leq & ||A_n(f_1-g_1)|| + \frac{\epsilon}{6} + ||A_m(f_1-g_1)|| \\
			      & \leq & \frac{\epsilon}{2}, \end{eqnarray*} 
			      
\noindent where the third inequality is because $||A_n(f_1-g_1)|| \leq ||f_1-g_1||_{\infty}$ and $||A_m(f_1-g_1)|| \leq ||f_1-g_1||_{\infty}$.  Similarly, $||A_n(f_2) - A_m(f_2)|| \leq \frac{\epsilon}{2}$.  Then if $n, m \geq N$, 
\begin{eqnarray*} 
||A_n(f) - A_m(f)|| & = & ||A_n(f_1 + if_2) - A_m(f_1 +if_2)|| \\
				  & = & ||(A_n(f_1) - A_m(f_1)) +i(A_n(f_2) + A_m(f_2))|| \\
				  & \leq & ||A_n(f_1) - A_m(f_1)|| + ||A_n(f_2) - A_m(f_2)|| \\
				  & \leq & \epsilon. \end{eqnarray*}
				  
\noindent This proves the claim.  \\

In particular, the fact that $\{\int f dA_n\}_{n=1}^{\infty}$ is Cauchy in the operator norm implies the following:  if $g,h \in \mathcal{H}$ and $f \in C(X)$, the sequence of complex numbers $\{\int_X f dA_{n_{g,h}}\}_{n=1}^{\infty}$ is Cauchy.  This is because we have the bound
$$\left| \int_X  f dA_{n_{g,h}} - \int_X f dA_{m_{g,h}} \right| = \left| \left \langle \left( \int f dA_n - \int f dA_m \right) g, h \right \rangle \right| \leq $$
$$\left| \left| \int f dA_n - \int f dA_m \right| \right| ||g|| ||h||,$$ 

\noindent and the last term goes to zero as $m$ and $n$ approach infinity.  

For $g, h \in \mathcal{H}$, define $\mu_{g,h}: C(X) \rightarrow \mathbb{C}$ by $f \mapsto \lim_{n \rightarrow \infty} \int f dA_{n_{g,h}}$, which is well defined by the above discussion, and since $\mathbb{C}$ is complete.  Observe that $\mu_{g,h}$ is a bounded linear functional.  We will show that it is bounded, and leave the proof of linearity to the reader.  Let $f \in C(X)$.  Then
$$|\mu_{g,h}(f)| = \left| \lim_{n \rightarrow \infty} \int_X f dA_{n_{g,h}} \right| =  \lim_{n \rightarrow \infty} \left|\int_X f dA_{n_{g,h}} \right|.$$

\noindent Now for all $n$
$$\left|\int_X f dA_{n_{g,h}} \right| \leq \int_X |f| d|A_{n_{g,h}}| \leq ||f||_{\infty} ||g||||h||,$$ 

\noindent and hence
 $$\lim_{n \rightarrow \infty} \left|\int_X f dA_{n_{g,h}} \right| \leq ||f||_{\infty} ||g||||h||.$$
 
 \noindent This shows that $\mu_{g,h}$ is bounded by $||g||||h||$.  We can now invoke Theorem \ref{riesz1} to conclude that $\mu_{g,h}$ is a measure.  
 
The map $[g,h] \in \mathcal{H} \times \mathcal{H} \mapsto \mu_{g,h}$ is sesquilinear.  Indeed, we will show that $[g,h] \mapsto \mu_{g,h}$ is linear in the first coordinate.  The remaining properties of sesquilinearity are proved with a similar approach, and are left to the reader. 

Let $g,h,k \in \mathcal{H}$, and let $f \in C(X)$.  Then
\begin{eqnarray*} 
\int_X f d\mu_{g+h,k} & = & \lim_{n \rightarrow \infty} \int_X f dA_{n_{g+h,k}} \\
				       & = & \lim_{n \rightarrow \infty} \left( \int_X f dA_{n_{g,k}} + \int_X f dA_{n_{h,k}} \right) \\
				       & = & \int_X f d\mu_{g,k} + \int_X f d\mu_{h,k}, 			    \end{eqnarray*}
				       
\noindent where the second equality is because of Remark \ref{Asesqui}.  

Consider a closed subset $C \subseteq X$, and choose a sequence of functions $\{f_m\}_{m=1}^{\infty} \subseteq C(X)$ such that $f_m \downarrow {\bf{1}}_C$ pointwise. For instance, we could let $f_m(x) = \max\{1 - md(x,C), 0\}.$ By the dominated convergence theorem, 
\begin{eqnarray*} 
\int_X {\bf{1}}_C d\mu_{g+h,k} & = & \lim_{m \rightarrow \infty} \int_X f_m d\mu_{g+h,k} \\
					        & = & \lim_{m \rightarrow \infty} \left( \int_X f_m d\mu_{g,k} + \int_X f_m d\mu_{h,k} \right) \\
					        & = & \int_X {\bf{1}}_C d\mu_{g,k} + \int_X {\bf{1}}_C d\mu_{h,k}.  \end{eqnarray*} 
					        
 \noindent Hence, for any closed $C \subseteq X$
 
 \begin{equation} \label{closed} \mu_{g+h,k} (C) = \mu_{g,k}(C) + \mu_{h,k}(C). \end{equation} 
 
By decomposing the measures $\mu_{g+h,k}, \mu_{g,k}, \mu_{h,k}$ into their real and imaginary parts, we can show that (\ref{closed}) is equivalent to the following:
\begin{equation} \label{part1}  \mathrm{Re}\mu_{g+h,k} (C) = \mathrm{Re}\mu_{g,k}(C) + \mathrm{Re}\mu_{h,k}(C), \end{equation} 
 
 \noindent and
\begin{equation} \label{part2}  \mathrm{Im}\mu_{g+h,k} (C) = \mathrm{Im}\mu_{g,k}(C) + \mathrm{Im}\mu_{h,k}(C). \end{equation}

By further decomposing $ \mathrm{Re}\mu_{g+h,k},  \mathrm{Re}\mu_{g,k},  \mathrm{Re}\mu_{h,k}$ into their positive and negative parts (denoted $\mathrm{Re}\mu_{g+h,k}^{+}$ and $\mathrm{Re}\mu_{g+h,k}^{-}$ respectively), we can show, by rearranging terms, that (\ref{part1}) is equivalent to 
\begin{equation} \label{seconde} M_1(C) = M_2(C), \end{equation}

\noindent where $M_1= \mathrm{Re}\mu_{g+h,k}^{+} +  \mathrm{Re}\mu_{g,k}^{-} +  \mathrm{Re}\mu_{h,k}^{-},$ and $M_2 = \mathrm{Re}\mu_{g+h,k}^{-} +  \mathrm{Re}\mu_{g,k}^{+} +  \mathrm{Re}\mu_{h,k}^{+}.$

Since $M_1$ and $M_2$ are positive Borel measures on a metric space, $M_1$ and $M_2$ are regular (see Remark \ref{regular}).  That is, we can conclude that $M_1(\Delta) = M_2(\Delta)$ for any Borel subset $\Delta \in \mathscr{B}(X)$.  By invoking the equivalence of (\ref{part1}) and (\ref{seconde}), we have that (\ref{part1}) is true for all $\Delta\in \mathscr{B}(X)$.  A similar approach, will yield that (\ref{part2}) is true for all $\Delta \in \mathscr{B}(X)$.  Hence, (\ref{closed}) is true for all $\Delta \in \mathscr{B}(X)$.  This shows linearity in the first coordinate.  As mentioned above, the following additional properties listed below are proved similarly: 

\begin{itemize}

\item  Let $g,h,k \in \mathcal{H}$.  Then $\mu_{g,h+k} = \mu_{g,h} + \mu_{g,k}.$

\item Let $\alpha \in \mathbb{C}$ and $g,h\in \mathcal{H}$.  Then $\mu_{\alpha g, h} = \alpha\mu_{g,h}.$

\item Let $\beta \in \mathbb{C}$ and $g,h\in \mathcal{H}$.  Then $\mu_{g, \beta h} = \overline{\beta} \mu_{g,h}.$

\end{itemize}

\noindent Hence, the map $[g,h] \mapsto \mu_{g,h}$ is sesquilinear.  We also note that $\mu_{g,h}$ inherits the following three additional properties: 

\begin{itemize} 

\item For $h \in \mathcal{H}$, $\mu_{h,h}$ is a positive Borel measure on $X$. 

\item For $g,h \in \mathcal{H}$, $\mu_{g,h}$ has total variation less than or equal to $||g||||h||$. 

\item For $g,h \in \mathcal{H}$, $\overline{\mu_{g,h}} = \mu_{h,g}$. 

\end{itemize}

\noindent We will spend a short time justifying the second item in the above list.  Suppose that $\Delta_1, ..., \Delta_n$ is a collection of disjoint subsets of $\mathscr{B}(X)$.  Then using a generalized Schwarz inequality for positive sesquilinear forms we calculate that
$$\sum_{k=1}^n |\mu_{g,h}(\Delta_k)| \leq \sum_{k=1}^n \left(\mu_{g,g}(\Delta_k) \mu_{h,h}(\Delta_k) \right)^{\frac{1}{2}} \leq $$
$$\left( \sum_{k=1}^n\mu_{g,g}(\Delta_k) \sum_{k=1}^n\mu_{h,h}(\Delta_k) \right)^{\frac{1}{2}} = (\mu_{g,g}(X) \mu_{h,h}(X))^{\frac{1}{2}} = ( ||g||^2 ||h||^2 )^{\frac{1}{2}} = ||g||||h||,$$

\noindent which shows that the total variation of $\mu_{g,h}$ is less than or equal to $||g||||h||$. 

Let $\Delta \in \mathscr{B}(X)$.  The map $[g,h] \mapsto \int_X {\bf{1}}_{\Delta} d\mu_{g,h}$ is a bounded sesquilinear form with bound $1$.  Indeed, 
$$|[g,h]| \leq ||{\bf{1}}_{\Delta}||_{\infty} ||g||||h|| = ||g||||h||.$$

By Theorem \ref{boundedsesqui}, there exists a unique bounded operator, $A(\Delta) \in \mathcal{B}(\mathcal{H})$, such that for all $g,h \in \mathcal{H}$
$$\langle A(\Delta)g, h \rangle = \int_X {\bf{1}}_{\Delta}  d\mu_{g,h},$$

\noindent with $||A(\Delta)|| \leq 1$.  Accordingly, define $A: \mathscr{B}(X) \rightarrow \mathcal{B}(\mathcal{H})$ by $\Delta \mapsto A(\Delta)$, and note that for $g, h \in \mathcal{H}$, $A_{g,h} = \mu_{g,h}$.

\begin{cla} $A$ is a positive operator-valued measure. 
\end{cla} 

\noindent Proof of claim: 

\begin{enumerate} 

\item Let $\Delta \in \mathscr{B}(X)$, and $h \in \mathcal{H}$.  Then 
$$\langle A(\Delta)h,h \rangle = \int_X {\bf{1}}_{\Delta}  d\mu_{h,h} \geq 0.$$ 

\noindent Hence, $A(\Delta)$ is a positive operator.

\item Let $h \in \mathcal{H}$.  Then
$$\langle A(X)h, h \rangle = \int_X d\mu_{h,h} = \mu_{h,h}(X) = \langle h, h \rangle,$$ and 
$$\langle A(\emptyset)h, h \rangle = \int_X {\bf{1}}_{\emptyset} d\mu_{h,h} = \mu_{h,h}(\emptyset) = 0.$$

\noindent Hence, $A(X) = \text{id}_{\mathcal{H}}$ and $A(\emptyset) = 0.$  

\item If $\{\Delta_n\}_{n=1}^{\infty}$ are pairwise disjoint sets in $\mathscr{B}(X)$, then for all $g,h \in \mathcal{H},$
$$\left \langle A \left(\bigcup_{n=1}^{\infty} \Delta_n \right)g , h \right \rangle = \int_X {\bf{1}}_{\bigcup_{n=1}^{\infty} \Delta_n} d\mu_{g,h} = \sum_{n=1}^{\infty} \mu_{g,h} (\Delta_n) =$$
$$\sum_{n=1}^{\infty} \int_X {\bf{1}}_{\Delta_n} d\mu_{g,h} = \sum_{n=1}^{\infty}\langle A(\Delta_n)  g, h \rangle. $$

\end{enumerate} 

\noindent This completes the proof of the claim.

We will now show that $A_n \rightarrow A$ in the $\rho$ metric.     Let $\epsilon > 0$.  Choose an $N$ such that for $n,m \geq N$, $\rho(A_n,A_m) \leq \epsilon.$  Let $ f \in \text{Lip}_1(X)$.  If $n \geq N$, and $h \in \mathcal{H}$ with $||h||=1$,
\begin{eqnarray*} 
\left| \left \langle \left( \int f dA_n - \int f dA \right) h, h \right \rangle \right| & = & \left| \int_X f dA_{n_{h,h}} - \int_X f dA_{h,h} \right| \\
& = & \lim_{m \rightarrow \infty}  \left| \int_X f dA_{n_{h,h}} - \int_X f dA_{m_{h,h}} \right| \\
& = & \lim_{m \rightarrow \infty} \left| \left \langle \left( \int f dA_n - \int f dA_m \right)h, h \right \rangle \right|,  \\
\end{eqnarray*} 

\noindent where the second equality is because $\int f dA_{h,h} = \mu_{h,h}(f) = \lim_{m \rightarrow \infty} \int f dA_{m_{h,h}}$.  For $m \geq N$
\begin{eqnarray*} 
\left| \left \langle \left( \int f dA_n - \int f dA_m \right)h, h \right \rangle \right| & \leq & \left| \left| \int f dA_n - \int f dA_m \right| \right| ||h||^2 \\
& = & \left| \left| \int f dA_n - \int f dA_m \right| \right| \\
& \leq & \rho(A_n,A_m) \\
& \leq & \epsilon.
\end{eqnarray*} 

\noindent Hence
$$\lim_{m \rightarrow \infty} \left| \left \langle \left( \int f dA_n - \int f dA_m \right)h, h \right \rangle \right| \leq \epsilon,$$ 

\noindent and therefore 
$$\left| \left| \int f dA_n - \int f dA \right| \right| \leq \epsilon.$$

\noindent Since the choice of $N$ is independent of $f \in \text{Lip}_1(X)$, $\rho(A_n, A) \leq \epsilon$, which shows that the metric space $(S(X), \rho)$ is complete.

\end{proof}

\subsection{A Fixed POVM Associated to an IFS} 

Let $X$ be the attractor set associated to an arbitrary IFS (with possibly essential overlap), and consider the Hilbert space $L^2(X, \mu)$, where $\mu$ is the Hutchinson measure. Recall that the maps

\begin{center} $F_i: L^2(X,\mu) \rightarrow L^2(X, \mu)$ given by $\displaystyle{\phi \mapsto \frac{1}{\sqrt{N}} (\phi \circ \sigma_i)}$ \end{center} 

\noindent for all $i = 0, ..., N-1,$ satisfy the operator identity

$$\sum_{i=0}^{N-1} F_i^*F_i = \text{id}_{\mathcal{H}}.$$ 

\noindent This is Theorem \ref{Kornelson} stated in the introduction. If we define $S_i = F_i^*$, we can rewrite the above operator identity as 

\begin{equation} \sum_{i=0}^{N-1} S_iS_i^* = \text{id}_{\mathcal{H}}.\end{equation} 

As desired, the following theorem generalizes the map $U$ in Theorem \ref{U} to positive operator-valued measures, in the case of an arbitrary IFS. 

\begin{thm} \label{contraction} [Davison] The map $V: S(X) \rightarrow S(X)$ given by $$B(\cdot) \mapsto \sum_{i=0}^{N-1} S_i B(\sigma_i^{-1}(\cdot))S_i^*$$ 

\noindent is a Lipschitz contraction in the $\rho$ metric. 

\end{thm}

\begin{proof} The proof of this theorem follows essentially the same line of reasoning as in the proof of Theorem \ref{U}. The correct version of the proof of Theorem \ref{U} can be found in \cite{Davison2}. At present, the only item we need to show is that $V$ is well defined as a map on $S(X)$. Indeed, let $B \in S(X),$ and $\Delta \in \mathscr{B}(X)$. Then $V(B)(\Delta)$ is a positive operator on $\mathcal{H}$. That is, if $h \in \mathcal{H}$, 

\begin{eqnarray*} 
\langle B(\Delta)h, h \rangle & = & \left \langle \left( \sum_{i=0}^{N-1} S_i B(\sigma_i^{-1}(\Delta))S_i^* \right) h, h \right \rangle \\
			      & = &  \sum_{i=0}^{N-1} \left \langle  S_i B(\sigma_i^{-1}(\Delta))S_i^*  h, h \right \rangle \\
			      & = & \sum_{i=0}^{N-1} \left \langle  B(\sigma_i^{-1}(\Delta))S_i^*  h, S_i^* h \right \rangle \\
			      & \geq & 0,
\end{eqnarray*} 

\noindent since $B(\sigma_i^{-1}(\Delta))$ is a positive operator for all $0 \leq i \leq N-1$. We leave it to the reader to verify that $B$ satisfies the remaining properties of a POVM.

\end{proof}

\begin{cor} \label{FixedPOVM} [Davison] By the Contraction Mapping Theorem, there exists a unique positive operator-valued measure, $A \in S(X)$, such that

\begin{equation} \label{projfixed} A(\cdot) = \sum_{i=0}^{N-1} S_i A(\sigma_i^{-1}(\cdot)) S_i^*. \end{equation}

\end{cor} 

\subsection{Dilation of the Fixed POVM}

We begin with some preliminary facts and definitions. Let $N \in \mathbb{N}$ such that $N \geq 2$, let $\mathcal{S} = \{\sigma_0,...,\sigma_{N-1}\}$ be an IFS (with possibly essential overlap) whose attractor set is $X$, and let $F_i: L^2(X,\mu) \rightarrow L^2(X,\mu)$ be as defined above. Recall that we previously defined $\Gamma_N = \{0,...,N-1\}$. If we let $\Omega = \prod_{1}^{\infty} \Gamma_N,$ it is well known that $\Omega$ is a compact metric space. The metric $m$ on $\Omega$ is given by

$$m(\alpha, \beta) = \frac{1}{2^j}$$ 

\noindent where $\alpha, \beta \in \Omega,$ and $j \in \mathbb{N}$ is the first entry at which $\alpha$ and $\beta$ differ. 

We next define the shift maps on this compact metric space. Indeed, for $0 \leq i \leq N-1$, let $\eta_i: \Omega \rightarrow \Omega$ be given by $\eta_i((\alpha_1, \alpha_2,...,)) = (i, \alpha_1, \alpha_2,,...,)$, and define $\eta: \Omega \rightarrow \Omega$ given by $\eta((\alpha_1,\alpha_2, \alpha_3,...)) = (\alpha_2, \alpha_3,....,)$. 

\begin{itemize} 

\item The maps $\eta_i$ are Lipschitz contractions on $\Omega$ in the $m$ metric, and therefore, the family of maps $\mathcal{T} = \{\eta_0,...,\eta_{N-1}\}$ constitutes an IFS on $\Omega$. 

\item The compact metric space $\Omega$ is itself the attractor set associated to the IFS $\mathcal{T}$.

\item The Hutchinson measure $P$ on $\Omega$ associated to the IFS $\mathcal{T}$ is called the Bernoulli measure, and it satisfies

$$P(\cdot) = \frac{1}{N} \sum_{i=0}^{N-1} P(\eta_i^{-1}(\cdot)).$$

\item The map $\eta$ is a left inverse for each $\eta_i$, meaning that $\eta \circ \eta_i = \text{id}_{\Omega}$ for each $0 \leq i \leq N-1.$ 

\end{itemize} 

Since the IFS $\mathcal{T}$ is disjoint, for each $0 \leq i \leq N-1$ we can define $T_i: L^2(\Omega,P) \rightarrow L^2(\Omega,P)$ by 

$$\phi \mapsto \sqrt{N} (\phi \circ \eta) {\bf{1}}_{\eta_i(X)},$$

\noindent and its adjoint $T_i^*: L^2(\Omega,P) \rightarrow L^2(\Omega, P)$ by 

$$\phi \mapsto \frac{1}{\sqrt{N}} (\phi \circ \eta_i),$$

\noindent such that the family of operators $\{T_i\}_{i=0}^{N-1}$ satisfies the Cuntz relations. Consequently, there exists a unique projection-valued measure, $E$, with respect to the pair $(\Omega, L^2(\Omega,P))$ such that 

\begin{equation} \label{pvm_alphabet} E(\cdot) = \sum_{i=0}^{N-1} T_iE(\eta_i^{-1}(\cdot))T_i^*. \end{equation} 

\noindent We also have from Corollary \ref{FixedPOVM} that there exists a unique POVM, A, with respect to the pair $(X, L^2(X,\mu))$ such that 

\begin{equation} \label{povm} A(\cdot) = \sum_{i=0}^{N-1} S_iA(\sigma_i^{-1}(\cdot)S_i^*, \end{equation} 

\noindent where $S_i = F_i^*.$ 

For each $\alpha \in \Omega$, define $\pi(\alpha) = \cap_{n=1}^{\infty} \sigma_{\alpha_1} \circ ... \circ \sigma_{\alpha_n}(X),$ where $\alpha = (\alpha_1, \alpha_2, ..., \alpha_n,...)$. Since the maps $\sigma_i$ are all contractive, $\pi(\alpha)$ is a single point in $X$. Define the map $\pi: \Omega \rightarrow X$ by $\alpha \rightarrow \pi(\alpha)$ as the coding map. 

\begin{lem} \cite{Kornelson} [Jorgensen, Kornelson, Shuman] The coding map is continuous. Moreover, for all $0 \leq i \leq N-1,$ we have the relation 

\begin{equation} \label{relation} \pi \circ \eta_i = \sigma_i \circ \pi. \end{equation} 

\end{lem}

We now are prepared to state a result from Jorgensen and his colleagues, which we will use in our below discussion. 

\begin{thm} \cite{Kornelson} \label{intertwining} [Jorgensen, Kornelson, Shuman]

\begin{enumerate} 

\item The operator $V: L^2(X, \mu) \rightarrow L^2(\Omega, P)$ given by

$$V(f) = f \circ \pi$$

\noindent is isometric. 

\item The following intertwining relations hold: 

$$VF_i = T_i^*V,$$

\noindent for all $0 \leq i \leq N-1$. 

\end{enumerate} 

\end{thm} 

Consider now the projection-valued measure $E(\pi^{-1}(\cdot))$ from the Borel subsets of $X$ into the projections on $L^2(\Omega, P)$. We have the following result, which will show that $E(\pi^{-1}(\cdot))$ is indeed a dilation of the POVM $A$, in the sense of Naimark's dilation theorem. 

\begin{thm} \text{} [Davison] The projection-valued measure $E(\pi^{-1}(\cdot)),$ and the positive operator-valued measure $A$ are related as follows: 

$$V^{*}E(\pi^{-1}(\cdot))V = A(\cdot).$$

\end{thm} 

\begin{proof} Define $L = V^*E(\pi^{-1}(\cdot))V,$ and observe that $L$ is a POVM with respect to the pair $(X, L^2(X,\mu))$. Our goal is to show that $L = A$. To this end, note that by the intertwining relations of Theorem \ref{intertwining}, we have that

\begin{eqnarray*} 
\sum_{i=0}^{N-1} F_i^*L(\sigma_i^{-1}(\cdot))F_i & = & \sum_{i=0}^{N-1} F_i^*V^* E(\pi^{-1}(\sigma_i^{-1}(\cdot)))VF_i \\
			      & = & \sum_{i=0}^{N-1} V^*T_i E(\pi^{-1}(\sigma_i^{-1}(\cdot))) T_i^*V \\
			      & = & \sum_{i=0}^{N-1} V^*T_i E(\eta_i^{-1}(\pi^{-1}(\cdot))) T_i^*V \\
			      & = & V^* \left( \sum_{i=0}^{N-1} T_i E(\eta_i^{-1}(\pi^{-1}(\cdot))) T_i^* \right) V \\
			      & =& V^*E(\pi^{-1}(\cdot))V \\
			      & = & L(\cdot), \end{eqnarray*}
			     
\noindent where the third equality is by equation (\ref{relation}), and the fifth equality is by equation (\ref{pvm_alphabet}).  Now $A$ is the unique POVM that satisfies equation (\ref{povm}). By the above computation, we see that $L$ also satisfies equation (\ref{povm}), and therefore, $L = A$.

\end{proof}

\section{Acknowledgements:}

The author would like to thank Judith Packer (University of Colorado) for her careful review of this material, and her guidance on this research.


\begin{thebibliography}{1}

\bibitem{Bistrom} Akerlund-Bistrom, C., ``A generalization of Hutchinson distance and applications,'' Random Computational Dynamics, 5, No. 2-3, 159-176 (1997).

\bibitem{Barnsley} Barnsley M., {\em{Fractals Everywhere}} (Second Edition), Academic Press, USA (1993).

\bibitem{Billingsley} Billingsley P. P., {\em{Convergence of Probability Measures}} (Second Edition), Wiley, New York, (1999).

\bibitem{Bogachev} Bogachev, V. {\em{Measure Theory: Volume II}}, Springer, New York, (2000). 

\bibitem{Conway} Conway, J., {\em{A Course in Functional Analysis}} (Second Edition), Springer, New York, 2000. 

\bibitem{Davison} Davison, T.,``Generalizing the Kantorovich Metric to Projection-Valued Measures,'' Acta Applicandae Mathematicae, DOI: 10.1007/s10440-014-9976-y (2014).

\bibitem{Davison2} Davison, T. ``Erratum to: Generalizing the Kantorovich Metric to Projection-Valued Measures,'' Acta Applicandae Mathematicae, DOI: 10.1007/s10440-015-0018-1 (2015). 

\bibitem{Davison3} Davison, T., ``Generalizing the Kantorovich Metric to Projection-Valued Measures: With an Application to Iterated Function Systems,'' University of Colorado at Boulder, ProQuest Dissertations Publishing (2015).

\bibitem{Hutchinson} Hutchinson J., ``Fractals and self similarity,'' Indiana University Mathematics Journal, 30, No. 5, 713-747 (1981). 

\bibitem{Kornelson} Jorgenson, P. Kornelson, K., and Shuman, K. ``Harmonic Analysis of Iterated Function Systems with Overlap.'' J. Math. Phys., 48(8):083511, 35 (2007).

\bibitem{Jorgensen3} Jorgensen, P., ``Iterated Function Systems, Representations, and Hilbert Space,''  Int. J. Math., 15, 813 (2004).

\bibitem{Jorgensen4} Jorgensen, P. ``The measure of a measurement.'' J. Math. Phys. 48 (10), 103506, 15 (2007). 

\bibitem{Jorgensen2} Jorgensen, P., ``Measures in Wavelet Decompositions,'' Adv. in Appl. Math., 34, No. 3 , 561-590 (2005). 

\bibitem{Jorgensen} Jorgensen, P., ``Use of Operator Algebras in the Analysis of Measures from Wavelets and Iterated Function System,'' Operator Theory, Operator Algebras, and Applications, Contemp. Math., 414, 13 - 26, Amer. Math. Soc. (2006). 

\bibitem{Kravchenko} Kravchenko, A. S., ``Completeness of the space of separable measures in the Kantorovich-Rubinshtein metric,'' Siberian Mathematical Journal, 47, No. 1, 68-76 (2006). 

\bibitem{Kribs} Kribs, D.W., ``A quantum computing primer for operator theorists," Linear Algebra Appl.,
400:147-167 (2005).

\bibitem{Kribs2}  Kribs, D.W., Laflamme, R., Poulin, D., and Lesosky, M., ``Operator
quantum error correction." Quantum Inf. Comput., 6(4-5):382-398 (2006).

\bibitem{Werner} Werner, R.F., "The Uncertainty Relation for Joint Measurement of Position and Momentum," Journal of Quantum Information and Computation, 4, No. 6, 546-562 (2004).  


\end{thebibliography}
\end{document}